\documentclass[12pt]{amsart}
\usepackage{fullpage,url,amssymb,amsmath,graphicx}

\usepackage{amsmath, amsthm, amssymb}
\usepackage{url}
\usepackage{graphicx}
\usepackage {amsmath}
\usepackage{amsthm}
\usepackage {amssymb}
\usepackage{amsmath}
\usepackage {amssymb}
\usepackage {graphicx,enumerate,booktabs}
\usepackage {color, tikz}
\usepackage[all]{xy}

\newtheorem{definition}{Definition}
\newtheorem{lemma}{Lemma}[section]
\newtheorem{theorem}[lemma]{Theorem}
\newtheorem*{theorem*}{Theorem}

\newtheorem{claim*}{Claim}
\newtheorem{remark}[lemma]{Remark}
\newtheorem{quest}[lemma]{Question}

\numberwithin{equation}{section}

\newcommand{\codim}{\operatorname{codim}}

\newcommand{\Osh}{{\mathcal O}}

\newcommand{\Hom}{\operatorname{Hom}}

\newcommand{\dmin}{\operatorname{dmin}}

\newcommand{\PP}{\mathbb{P}}
\newcommand{\ZZ}{\mathbb{Z}}
\newcommand{\RR}{\mathbb{R}}
\newcommand{\FF}{\mathbb{F}}

\definecolor{cof}{RGB}{219,144,71}
\definecolor{pur}{RGB}{186,146,162}
\definecolor{greeo}{RGB}{91,173,69}
\definecolor{greet}{RGB}{52,111,72}

\begin{document}

\title{Dual Toric Codes and Polytopes of Degree One}
\author[V.~Gauthier]{Val\'erie Gauthier Uma\~na} \address{Valerie Gauthier Uma\~na\\ Facultad de Ciencias y Matem\'aticas, Departamento
  de Matem\'aticas\\ Universidad del Rosario\\ Bogot\'a\\ Colombia}
\email{{gauthier.valerie@urosario.edu.co}}

\author[M.~Velasco]{Mauricio Velasco} \address{Mauricio Velasco\\ Departamento
  de Matem\'aticas\\ Universidad de los Andes\\ Carrera 1 No. 18a 10\\ Edificio
  H\\ Bogot\'a\\ Colombia}
\email{{mvelasco@uniandes.edu.co}}


\begin{abstract} We define a statistical measure of the typical size of short words in a linear code over a finite field. We prove that the dual toric codes coming from polytopes of degree one are characterized, among all dual toric codes, by being extremal with respect to this measure. We also give a geometric interpretation of the minimum distance of dual toric codes and characterize its extremal values. Finally, we obtain exact formulas for the parameters of both primal and dual toric codes associated to polytopes of degree one. 
\end{abstract}

\maketitle
\section{Introduction}
Following J. Hansen~\cite{H1} we can construct a linear code from any projective toric variety over a sufficiently large finite field $\FF_q$. A projective toric variety is specified by a lattice polytope $P$ and the construction of the corresponding {\it toric code}  and of its {\it dual toric code} depend only on this polytope and on the chosen field. More specifically,

\begin{definition} For any prime power $q$ such that $P\subseteq [0,q-1]^m\subseteq \RR^m$ define $t:=(q-1)^m$ and let $\FF_q$ be the finite field of size $q$. Let $T:={\rm Spec}(\FF_q[x_1^{\pm},\dots, x_m^{\pm}])$ be the $m$-dimensional algebraic torus over $\FF_q$. Define $H^0(P)$ to be the vector space over $\FF_q$ spanned by Laurent monomials in $x_1,\dots, x_m$ whose exponent vectors lie in $P$. The primal toric code $C_P(\FF_q)$ is the vector subspace of $\FF_q^t$ obtained by evaluating the polynomials in $H^0(P)$ at the points of $T(\FF_q)$ (i.e. at the $t$ points of $(\FF_q^*)^m$ in some fixed order). The dual toric code $C_P^*(\FF_q)$ is the vector subspace of $\Hom(\FF_q^t,\FF_q)$ consisting of linear functionals which vanish at all points of $C_P$.
\end{definition}

Primal toric codes are a special case of a construction introduced by Goppa in the early eighties~\cite{Goppa} for building codes from general algebraic varieties endowed with a line bundle and a $0$-cycle. Toric codes however, are of special interest since the choice of the $0$-cycle is canonical and because we can use polyhedral methods to control the resulting linear codes. At the same time, the rich geometry of toric varieties makes toric codes an interesting testing ground for ideas on the behavior of general algebraic codes. For these reasons, primal toric codes have been the focus of much work by several authors including Hansen~\cite{H1,H2}, Joyner~\cite{J}, Little and Schenck~\cite{LS}, Ruano~\cite{R}, Soprunov and Soprunova~\cite{SS} among others. 

Despite significant progress we believe that the fundamental questions in the area of toric codes are still open. In our view, these concern the determination of extremal toric codes. For an example of one such question, motivated by applications to error-correcting codes, consider
\begin{quest}
 Given a positive integer $m\geq 3$, and a prime power $q$, determine all full-dimensional lattice polytopes $P\subseteq [0,q-1]^m$ for which the minimum distance of the primal toric code $C_P(\FF_q)$ is as large as possible. 
\end{quest}
This article contributes to this line of inquiry by defining and characterizing some extremal properties of {\it dual toric} codes $C_P^*$. Our strategy consists of relating the parameters of the code $C_P^*$ with the geometry of the projective toric variety determined by $P$ with the aim of using the powerful classification theorems of toric geometry. To describe our results we introduce some notation, let $\phi: T\rightarrow \PP(H^0(P)^*)$ be the morphism determined by the Laurent monomials in the lattice polytope $P$. Define $X_P^{\circ}:=\phi(T)$ and let $X_P$ be its projective closure.

Our first Theorem gives a geometric interpretation for the minimum distance of dual toric codes,
\begin{theorem*} The following statements hold for any full-dimensional lattice polytope $P\subseteq [0,q-1]^m$ with at least $m+2$ lattice points,
\begin{enumerate}
\item{ The minimum distance of $C_P^*(\FF_q)$ equals the cardinality of the smallest set of points $S\subseteq X_P^\circ$ which does not impose independent conditions on linear forms in $\PP(H^0(P)^*)$.}
\item{ The minimum distance of $C_P^*(\FF_q)$ is at least three. It equals three iff $X_P^{\circ}$ contains three collinear $\FF_q$-points.}
\item{ For every sufficiently large $q$,  the minimum distance of $C_P^*(\FF_q)$ is at most $\codim(X_P)+2$. The upper bound is sharp. It is achieved if $X_P$ is of codimension one or if $P=[0,c+1]\subseteq \RR$.}
\end{enumerate}
\end{theorem*}

Part $(2)$ of the previous Theorem shows that the shortest words in a dual toric code $C_P^*$ always come from points of $X_P$ in special configuration. As a result, we expect such words to be comparatively rare. To get a more accurate sense of the properties of the code $C_P^*$ it may be preferable to take a statistical approach and to look for short, but {\it typical} words. To do so, we propose the following definition, applicable to any linear code $V\subseteq k^t$ with $k$ a finite field,

\begin{definition} For a set $S\subseteq \{1,\dots,t\}$ let $w_S$ be the weight (i.e. cardinality of the support) of the shortest word of $V$ whose support contains $S$. For an integer $s\in \{1,\dots,t\}$ we let the mode of size $s$ of $V$, denoted $M(s,V)$ to be the most frequent $w_S$ when $S$ ranges over all subsets of $\{1,\dots,t\}$ of size $s$. If $B\subseteq \{1,\dots,t\}$ is any fixed set then we define the mode of size $s$ relative to $B$ of $V$, denoted $M^B(s,V)$ as the most frequent $w_S$ when $S$ ranges over all subsets of $B$ of cardinality $s$.
\end{definition}
The main result of this article is a classification of polytopes $P$ whose dual toric codes are extremal, in the sense that their mode of size $\codim(X_P)+1$ is as large as possible. To state the Theorem precisely we need to recall the following definition due to Batyrev and Nill~\cite{BN},
\begin{definition} An $m$-dimensional lattice polytope $P\subseteq \RR^m$ is of degree one if its first dilate containing an interior lattice point is $m$. 
\end{definition}
Batyrev and Nill give in~\cite{BN} a complete classification of such polytopes via combinatorial arguments. Polytopes of degree one are classified as follows,
\begin{definition} For a nondecreasing sequence of integers $0<a_0\leq \dots \leq a_{m-1}$ define the Lawrence prism $L(a_0,\dots, a_{m-1})$ as the convex hull
\[{\rm Conv}\left(0, e_1,\dots, e_{m-1}, a_0 e_m, e_1+a_1e_m,\dots, e_{m-1}+a_{m-1}e_m \right)\subseteq \RR^m.\]
\end{definition}
\begin{theorem*}\cite[Theorem 2.5]{BN} Let $P$ be a lattice polytope. $P$ is of degree one if and only if it can be obtained as an iterated pyramid over either a Lawrence prism $L(a_0,\dots, a_{n-1})$ or over an exceptional simplex $\Delta_2:={\rm Conv}\left((0,0),(2,0),(0,2)\right)$.
\end{theorem*}

An alternative geometric proof of this classification, based on the classification of varieties of minimal degree of classical algebraic geometry is given in~\cite{BSV}. 
Polytopes of degree one satisfy well known extremal geometric and combinatorial properties (see Section~$\S$\ref{DegreeOne} for details). The following Theorem shows that polytopes of degree one are also characterized by an extremal statistical property of their dual toric codes (see Theorem~\ref{Extremality} and Theorem~\ref{typical} for more precise statements), 

\begin{theorem*} Let $P\subseteq \RR^m$ be any full-dimensional lattice polytope with at least $m+2$ lattice points and let $c:=\codim(X_P)$. For all but finitely many finite fields $\FF_q$ there is a cofinal subset $\mathcal{C}$ of the poset of finite fields containing $\FF_q$ such that, for every $\FF_r\in \mathcal{C}$, we have 
\[M^{X_P(\FF_q)}(c+1,C_P^*(\FF_r))\in\{c+2,c+3\}.\]
Moreover $M^{X_P(\FF_q)}(c+1,C_P^*(\FF_r))=c+3$ if and only if $P$ is a polytope of degree one.
\end{theorem*}

The previous Theorem shows that the dual toric codes of polytopes of degree one have a distinguished place among all toric codes. In Section~$\S\ref{DegreeOne}$ we focus on the determination of the parameters of dual and primal toric codes of polytopes of degree one. Our results on primal codes extend work of Little and Schenck~\cite{LS} from polygons to polytopes of all dimensions,

\begin{theorem*} Let $P$ be a polytope of degree one and dimension at least two. For all but finitely many finite fields $\FF_q$ we have that either $\dmin(C_P^*(\FF_q))=3$ or $P=\Delta_2$ and $\dmin(C_P^*(\FF_q))=4$. 
\end{theorem*}
\begin{theorem*} Let $P\subseteq [0,q-1]^m\subseteq \RR^m$ be a lattice polytope of degree one and let $(n,k,\dmin)$ be the length, rank and minimum distance of the code $C_P(\FF_q)$. Then $n=(q-1)^m$ and:
\begin{enumerate} 
\item{ If $P$ is obtained from $\Delta_2\subseteq \RR^2$ by iterated pyramids then 
\[\begin{array}{lll}
k=m+4 & \text{and }& \dmin=(q-1)^m-2(q-1)^{m-1}.\\
\end{array}
\]
}
\item{If for some $0< a_0\leq a_1\leq\dots \leq a_{n-1}$ $P$ is obtained from $L(a_0,\dots, a_{n-1})$  by iterated pyramids then
$k=m+\sum_{i=0}^{n-1}a_i$ and 
\[
\dmin=\begin{cases}
(q-1)^m-a_{n-1}(q-1)^{m-1}\text{ , if $a_{n-2}<a_{n-1}$}\\
(q-1)^m- (q-1)^{m-2}\Big((q-1) + a_{n-1} (q-2)\Big)\text{ , if $a_{n-2}=a_{n-1}$.}\\
\end{cases}
\]
}
\end{enumerate}
\end{theorem*}
Our proof of the last Theorem, which uses the toric code methods proposed by Soprunov and Soprunova~\cite{SS}, shows that the minimum distance is always achieved by a section which is completely reducible. This same phenomenon had been observed by Little and Schenck in the context of toric surface codes. We believe it is an interesting question to determine whether this is always the case for primal toric codes. 

{\bf Acknowledgements.}Ê We wish to thank Grigoriy Blekherman, Gregory G. Smith, Damiano Testa and Anthony Varilly-Alvarado for many useful conversations during the completion of this project. M. Velasco was partially supported by the FAPA funds from Universidad de los Andes.

{\bf Preliminaries and notation.} {\it Codes:} If $V$ is a vector space over a field $k$ we let $V^*:=\Hom(V,k)$ be its vector space dual. A linear code is a vector subspace $V\subseteq k^t$. The length of $V$ is $t$ and the rank of $V$ is $\dim(V)$. The elements $w\in V$ are called words. The support of a word is $w$ the set $S\subseteq \{1,\dots, t\}$ such that $w_j\neq 0$ and the weight of $w$ is $|S|$. The minimum distance of $V$, denoted $\dmin(V)$, is the smallest weight of a nonzero word in $V$. It coincides with $t-z$ where $z$ is the largest number of zero entries among all words of $V$. {\it Polytopes: } A lattice polytope is a polytope whose vertices have integer coordinates. If $P\subseteq \RR^m$ is a full-dimensional lattice polytope then the pyramid over $P$ is the lattice polytope $Q:={\rm Conv}(P\times \{0\} \cup \{(0,1)\})\subseteq \RR^m\times \RR$. We let $\Delta_f:={\rm Conv}(0,e_1,\dots,e_m)\subseteq \RR^m$. {\it Posets: }  A subset $S$ of a partially ordered set is cofinal if every element of the poset has an upper bound in $S$.

\section{Geometry of dual codes.}\label{DualToricCodes}
For clarity, we restate the definition of dual toric code. In this section we give an interpretation of the support of a word in the dual toric code as a set of points in $X_P^\circ$ and use this interpretation to characterize the extremal values of the minimum distance of such codes.

We fix a full-dimensional lattice polytope $P\subseteq [0,q-1]^m$ and a finite field $\FF_q$. 
\begin{definition} Let ${\rm Fun}(T(\FF_q),\FF_q)$ be the vector space of $\FF_q$-valued functions on the $\FF_q$-points of the $m$-dimensional algebraic torus. The primal code $C_P(\FF_q)$ is the subset of ${\rm Fun}(T(\FF_q),\FF_q)$ obtained by sending each Laurent polynomial in $H^0(P)$ to its corresponding function (the assumption $P\subseteq [0,q-1]^m$ guarantees that this identification is injective). The dual toric code is the vector subspace
\[ C_P^*(\FF_q):=\{\phi \in {\rm Fun}(T(\FF_q),\FF_q)^{*}:\forall w\in C_P(\FF_q)\left(\phi(w)=0\right)\}.\]
The vector space ${\rm Fun}(T(\FF_q),\FF_q)^{*}$ contains a {\it point evaluation} $\ell_{\alpha}$ for each $\alpha\in T(\FF_q)$ which maps a function to its value at the point $\alpha$. These point evaluations are a basis for the space ${\rm Fun}(T(\FF_q),\FF_q)^{*}$. We define the support of a word $\phi\in {\rm Fun}(T(\FF_q),\FF_q)^{*}$ as the set of points in $T(\FF_q)$ which appear with nonzero coefficient in the unique expression of $\phi$ as a linear combination of point evaluations $\ell_\alpha$ for $\alpha\in T(\FF_q)$. 
\end{definition}
\begin{definition} Let $A(\FF_q,P)$ be the matrix with columns indexed by the Laurent monomials in $P$ and rows indexed by the points of $T(\FF_q)$ whose entry $(\alpha,n)$ is the value of the monomial $n$ at the point $\alpha$. For a set $S\subseteq T(\FF_q)$ let $A_S(\FF_q,P)$ be the submatrix of $A$ consisting of the rows of $A$ indexed by the points of $S$. We will denote $A(\FF_q,P)$ (resp. $A_S(\FF_q,P)$ ) by $A$ (resp. $A_S$) when $\FF_q$ and $P$ are clear from the context. 
\end{definition}

\begin{lemma}\label{THM: Dual} For each subset $S\subseteq X_P^{\circ}$ let $r_S$ be the dimension of the vector space of words $\phi \in C_P^*$ whose support is contained in $S$. The following statements hold,
\begin{enumerate}
\item \label{l1P1} $r_S=|S|-{\rm rk}(A_S)$. In more classical language, $r_S$ measures the failure of $S$ to impose independent conditions on linear forms on $\PP(H^0(P)^*)$.
\item \label{l1P2} The minimum distance of the dual code $C_P^*$ is achieved by a word supported on the smallest set of points $S\subseteq X_P^{\circ}$ which does not impose independent conditions on linear forms.
\item \label{l1P3} If $f_S$ is the number of words of $C_P^*$ with support equal to $S$ then
\[ f_S=\sum_{A\subseteq S} (-1)^{|A|}q^{r_A}\]
\end{enumerate} 
\end{lemma}
\begin{proof} (\ref{l1P1}) The vector subspace of $C_P^*$ consisting of words whose support is contained in $S$ is, by definition, isomorphic to ${\rm Ker}(A_S^t)$. It follows that $r_S=|S|-{\rm rk}(A_S^t)$ so the equality holds since the rank of a matrix and that of its transpose coincide. A set of $j$ points imposes independent conditions on a vector space of forms $V$ if the set of forms in $V$ vanishing at $k$ of the points is a proper superset of those vanishing through any $k+1$ of them for $0\leq k\leq j-1$. As a result the difference $|S|-{\rm rank}(A_S)$ measures the failure of $S$ to impose independent conditions. (\ref{l1P2}) Follows from~(\ref{l1P1}). (\ref{l1P3}) Since the set of words of $C_P^*$ with support contained in a set $S$ form a vector space over $\FF_q$ of dimension $r_S$, there are $q^{r_S}$ such words. Partitioning them according to their support we see that, for any set $S$ the equality $q^{r_S}=\sum_{B\subseteq S} f_B$ holds. The claimed formula for $f_S$ follows by inclusion-exclusion. 
\end{proof}

\begin{theorem} For every full-dimensional lattice polytope $P\subseteq \RR^m$ such that $X_P\subset \PP(H^0(P)^*)$ the following statements hold,
\begin{enumerate}
\item{The minimum distance of $C_P^*(\FF_q)$ is at least three. It equals three iff $X_P^{\circ}$ contains three collinear $\FF_q$-points.}
\item{For every prime number $p$ there is a cofinal set $\mathcal{C}$ of the poset of finite fields containing $\FF_p$ such that, for every $\FF_r\in \mathcal{C}$ we have
\[ \dmin(C_P^*(\FF_r))\leq \codim(X_P)+3.\]
Moreover the bound is sharp. It is achieved by $P:=[0,\codim(X_P)+1]\subseteq \RR$. 
}
\end{enumerate} 
\end{theorem}
\begin{proof}$(1)$ Follows from Lemma~\ref{THM: Dual} part $(\ref{l1P2})$. $(2)$ By base-change we think of $X_P\subseteq \PP(H^0(P)^*)$ as a scheme over $\overline{\FF_p}$ and let $c:=\codim(X_P)$. Choose a set $S$ of $c+2$ general points of $X_P^{\circ}(\overline{\FF_p})$ and let $H$ be the projective subspace of dimension $c+1$ they span. Note that, by genericity of the points, $H\cap X_P$ has dimension one and $H\cap(X_P\setminus X_P^{\circ})$ is a finite set. As a result there exists an $\overline{\FF_p}$-point $\beta\in (X_P^{\circ}\cap H)\setminus S$. Since $S\cup \{\beta\}\subseteq H$ the points of $S\cup \{\beta\}$ do not impose independent conditions on linear forms. Let $\mathcal{C}$ be the cofinite set of the poset of finite fields containing $\FF_p$ consisting of fields which contain the coordinates of the points of $S\cup \{\beta\}$. By Lemma~\ref{THM: Dual} for every field $\FF_r\in \mathcal{C}$ we have
\[ \dmin(C_P^*(\FF_r))\leq c+3\]
as claimed. If $P=[0,c+1]\subseteq \RR$ then $X_P$ is the rational normal curve in $\PP^{c+1}$. Since the Vandermonde matrix is invertible it follows that every set of $c+2$ distinct points in $X_P$ are independent and thus $\dmin(C_P^*)\geq c+3$ over any field proving the sharpness of the bound.
\end{proof}

\subsection{Statistics of dual codes}\label{stats}

As shown in Lemma~\ref{THM: Dual}, the words of minimum weight of $C_P^*$ arise due to the existence of small sets of points of $X_P^{\circ}$ lying in linearly special position. This observation suggests that it would be interesting to look at dual codes following a more statistical approach in which we ask questions about ``most words" and not about all words. To do so, we propose the following,

\begin{definition} Let $k$ be a finite field and let $V\subseteq k^t$ be a linear code. For a set $S\subseteq \{1,\dots,t\}$ let $w_S$ be the weight of the shortest word of $V$ whose support contains $S$. For an integer $s\in \{1,\dots,t\}$ we let the mode of size $s$ of $V$, denoted $M(s,V)$ to be the most frequent $w_S$ as $S$ ranges over all subsets of $\{1,\dots,t\}$ of size $s$. If $B\subseteq \{1,\dots,t\}$ is any fixed set then we define the mode of size $s$ relative to $B$ of $V$, denoted $M^B(s,V)$ as the most frequent $w_S$ when $S$ ranges over all subsets of $B$ of cardinality $s$. The mode depends on the integer $s$ but to ease the notation we will often drop $s$ when it is clear form the context. 
\end{definition}

\begin{remark} To a fixed polytope $P$ and a pair of finite fields $\FF_r\supseteq \FF_q$ we can associate a pair of codes $C_P^*(\FF_r)\supseteq C_P^*(\FF_q)$ whose supports are subsets of $X_P(\FF_r)\supseteq X_P(\FF_q)$ respectively. As a result, for any positive integer $s$, we can speak about the relative mode $M^{X_P(\FF_q)^\circ}(s, C_P^*(\FF_r))$.
This quantity will be our main object of interest.
\end{remark}

In this section, we let $c:=\codim(X_P)$ and study the mode of size $s:=c+1$ of the dual toric codes $C_P^*(\FF_q)$, henceforth denoted $M(C_P^*(\FF_q))$. Our main result is that over most fields, the codes $C_P^*$ for which the $(c+1)$-mode is maximal are precisely those for which $P$ is a polytope of degree one (see Theorem~\ref{Extremality} below for a precise statement). Our method of proof shows that the $(c+1)$-mode $M(C_P^*)$ is closely related with the degree of the projective variety $X_P$.

To capture the behavior of most $(C+1)$-tuples, with the aim of understanding the mode, we introduce the following concept

\begin{definition} Let $c:=\codim(X_P)$. An element $\vec{\beta}=(\beta_1,\dots, \beta_{c+1}) \in X_P(\FF_q)$ is called a {\it generic $(c+1)$-tuple} if it satisfies the following properties,
\begin{enumerate}
\item{The projective space $H$ spanned by the points $\beta_i$ has dimension $c$.}
\item{ The dimension of $H\cap X_P$ is $0$.}
\item{ $H\cap X_P\subseteq X_P^{\circ}$.}
\end{enumerate}
We denote the set of generic tuples over $\FF_q$ as $\mathcal{G}(\FF_q)$.
\end{definition}
The terminology ``generic $(c+1)$-tuples" is justified by the following Lemma which shows that in most cases, strictly more than one half of all $(c+1)$-tuples are generic,

\begin{lemma} \label{mostGeneric} For all but finitely many finite fields $\FF_q$ we have
\[ \frac{|\mathcal{G}(\FF_q)|}{|X_P^{\circ}(\FF_q)|^{c+1}}>\frac{1}{2}\]
\end{lemma}
\begin{proof} Let $p$ be any prime number. We think of $X_P$ as an integral subscheme of $\PP^n$ over $\FF_p$. We denote its dimension by $m<n$ and its degree by $d$. Let $\mathcal{Z}\subseteq X_P^d\times {\rm Gr}(c+1,n+1)$ be the incidence correspondence given by $Z:=\{((\beta_1,\dots, \beta_d),A):\beta_i\in A \}$ and let $\pi:Z\rightarrow X_P^{c+1}$ be the projection to the first $c+1$ components. Note that $\pi$ is a proper and thus closed morphism.
We first show that the complement of the set of generic tuples is a proper closed subvariety of $X_P^{c+1}$. If $C_1\subseteq \mathcal{Z}$ is the closed subset of $\mathcal{Z}$ where the points $\beta_1,\dots, \beta_c$ are linearly dependent then $B_1:=\pi(C_1)$ is a closed set and moreover $B_1\subset X_P^{c+1}$ because $X_P$ is non-degenerate. If $C_2$ is the closed subset of $\mathcal{Z}$ where $\beta_d\in X_P\setminus X_P^{\circ}$ then $B_2:=\pi(C_2)$ is a proper closed set of $X_P^{c+1}$ since by~\cite[Theorem 6.3]{Jo} a general $c$-dimensional projective subspace will not intersect the variety $X_P\setminus X_P^{\circ}$ because its codimension is higher than $c$. If $B_3:=\{\vec{\beta}\in X_P^{c+1}: \dim(\pi^{-1}(\vec{\beta}))\geq 1\}$ then $B_3$ is closed by Chevalley's upper-semicontinuity of fiber dimension. $B_3$ is a proper closed subset of $X_P^{c+1}$ by~\cite[Theorem 6.3 ]{Jo} which shows that a general complementary subspace will intersect $X_P$ at a reduced set of points. We conclude that the set of non-generic tuples is contained in the proper closed set $B:=B_1\cup B_2\cup B_3$ which is defined over $\FF_p$.
Now, by Noether's normalization (see~\cite{Hochster} for a proof that this can be done over the base field) each irreducible component of $B$ admits a finite morphism to $\PP^j$ over $\FF_p$ for some $j\leq \dim(B)\leq m(c+1)-1$. It follows that the number of points of $B$ over every field $\FF_s$ containing $\FF_p$ is at most $O(s^{m(c+1)-1})$. By the same argument $|X_P^{\circ}(\FF_s)^{c+1}|$ is $O(s^{m(c+1)})$. As a result, as the size $s$ of the field increases, the proportion of non-generic tuples goes to zero independently of $p$, proving the claim.\end{proof}

To state the main Theorem of this section we briefly recall the definition of varieties of minimal degree. Let $X\subset \PP^m$ be a non-degenerate projective variety of codimension $c$ over an algebraically closed field. If $S\subseteq X$ is a set of $c+1$ general points of $X$ then $\langle S\rangle$ is a projective subspace of $\PP^m$ of dimension $c$ and $\langle X\rangle\cap S$ consists of  $\deg(X)$ reduced points. Since $S\subseteq X\cap \langle S\rangle$, $\deg(X)\geq c+1$. $X$ is a variety of minimal degree if $\deg(X)=c+1$. The classification of varieties of minimal degree, due to Del Pezzo~\cite{DP} and Bertini~\cite{B} is one of the important achievements of classical algebraic geometry (see~\cite{EH} for a modern proof of the classification valid in all characteristics).

\begin{theorem}\cite[Theorem 1]{EH} If $X\subseteq \PP^m$ is a variety of minimal degree then it is a cone over a smooth such variety. If $X$ is smooth then either:
\begin{enumerate}
\item{ $X$ is a quadric hypersurface or}
\item{ $X$ is a rational normal scroll or}
\item{ $X$ is the Veronese surface in $\PP^5$.}
\end{enumerate}
\end{theorem}

We are now in a position to prove our main Theorem. Recall that $M(C_P)$ denotes the $c+1$-mode of the code $C_P$ where $c:=\codim(X_P)$, 
\begin{theorem}\label{Extremality} For all but finitely many fields $\FF_q$ the following statements hold:
\begin{enumerate}
\item{\label{MDeg} If $X_P$ is a variety of minimal degree (i.e. if $\deg(X_P)=c+1$) then
\[M(C_P^*(\FF_q))=c+3.\] }
\item{ \label{AMDeg} If $X_P$ is a variety of almost minimal degree (i.e. if $\deg(X_P)=c+2$) then
\[M(C_P^*(\FF_q))=c+2.\]
}
\item{\label{NMDeg} If $X_P$ is not a variety of minimal degree (i.e. if $\deg(X_P)\geq c+2$) then for all fields $\FF_{r}$ in a cofinal subset of the poset of finite fields containing $\FF_q$ we have
\[M^{X_P(\FF_q)^\circ}(C_P^*(\FF_r))=c+2.\]}
\end{enumerate}
\end{theorem}
\begin{proof} By Lemma~\ref{mostGeneric} we know that for all but finitely many finite fields strictly more than one half of all $(c+1)$-tuples of points in $X_P^{\circ}$ are generic tuples. We will determine the size of the shortest word whose support contains a generic tuple in the tree cases above. The size of this word will turn out to be independent of the chosen tuple and thus this length is equal to the mode of size $c+1$ of $C_P^*$.
Let $\beta_1,\dots, \beta_{c+1}$ be a generic tuple over $\FF_q$, let $S:=\{\beta_1,\dots, \beta_{c+1}\}$ and let $H:=\langle S\rangle$ be the projective subspace spanned by the points of $S$. Since the tuple is generic there are no words of $C_P^*$ whose support equals $S$.
(\ref{MDeg}) If $X_P$ is of minimal degree then $H\cap X_P=S$. As a result, any additional point $\beta_{c+2} \in X_P(\FF_q)$ is independent of the points of $S$ and there are no words of size $c+1$ in $C_P^*(\FF_q)$ whose support contains $S$.  On the other hand, for generic $\beta_{c+2}\in X_P^{\circ}(\FF_q)$, $\langle\beta_1,\dots, \beta_{c+2}\rangle\cap X_P$ is isomorphic to a non-degenerate curve of degree $c+1$ in $\PP^{c+1}$. Such a curve must be a rational normal curve and, since it contains $\FF_q$-points, must be isomorphic to $\PP^1$ over $\FF_q$. Thus we can choose an additional point $\beta_{c+3}\in X_P^\circ(\FF_q)$ such that $\beta_{c+3}\in \langle \beta_1,\dots, \beta_{c+2}\rangle$. By Lemma~\ref{THM: Dual} we conclude that there is a word of $C_P^*(\FF_q)$ of size $c+3$ whose support contains $S$ proving the claim.
$(\ref{AMDeg})$ If $\deg(X_P)=c+2$ then $H\cap X$ is a $0$-dimensional scheme defined over $\FF_q$. It consists of $S$ and an additional point $\beta_{c+2}$ and since the points of $S$ are defined over $\FF_q$ the same is true about $\beta_{c+2}$. It follows from Lemma~\ref{THM: Dual} that the shortest word whose support contains $S$ has size $c+1$ as claimed. $(\ref{NMDeg})$ The argument for part (\ref{NMDeg}) is similar to part (\ref{AMDeg}) with the difference that we no longer know the field extension required to define a point in $X_P\cap H \setminus S$. However, we know that $\Gamma:=H\cap X_P\setminus{S}$ is a $0$-dimensional scheme of degree $d=\deg(X)-(c+1)$. It follows that there is a point of $\Gamma$ in some field $\FF_{q' }$ with $[ \FF_{q' }: \FF_q]\leq d$ and thus that such a point exists in any field containing $\FF_{q^{d!}}$. As a result, for every field $\FF_r$ containing $\FF_{q^{d!}}$ (and in particular in a cofinal subset of all finite fields containing $\FF_q$) we see from Lemma~\ref{THM: Dual}  that the shortest word of $C_P^*(\FF_r)$ whose support contains any generic tuple in $X_P(\FF_q)$ is $c+2$ as claimed.
\end{proof}

\section{Toric codes from degree one polytopes.}\label{DegreeOne}

Little and Schenck characterize in~\cite[Section 3]{LS} the rank and minimum distances of the primal toric codes of all lattice polygons without interior lattice points. Such polygons were classified by Arkinstall in~\cite{A}. As mentioned in the introduction, Batyrev and Nill propose in~\cite{BN} a higher-dimensional analogue of such polytopes which they call polytopes of degree one. A full-dimensional lattice polytope $P\subseteq \RR^m$ is of degree one if its smallest multiple containing interior lattice points is $mP$. There are at least two points of view which have appeared in the literature under which polytopes of degree one constitute a natural family:
\begin{enumerate}
\item Enumerative: These are the only polytopes whose Erhardt series has a numerator of degree one, that is, for which there exists a number $h_1^*\neq 0$ such that
\[ \sum_{k=0}^{\infty} |kP\cap \ZZ^m|t^m=\frac{1+h_1^*t}{(1-t)^{m}}.\]
See~\cite[Main Theorem]{BN} for a proof. 
\item Geometric: Such polytopes are in one-to-one correspondence with projective toric varieties of minimal degree. See~\cite[Proposition 6.4, Remark 6.8]{BSV} for a proof.
\end{enumerate}
Explicitly, polytopes of degree one are classified as,
\begin{definition} For a nondecreasing sequence of integers $0<a_0\leq \dots \leq a_{m-1}$ define the Lawrence prism 
\[L(a_0,\dots, a_{m-1})={\rm Conv}\left(0, e_1,\dots, e_{m-1}, a_0 e_m, e_1+a_1e_m,\dots, e_{m-1}+a_{m-1}e_m \right).\]
\end{definition}
\begin{theorem}\cite[Theorem 2.5]{BN} Let $P$ be a lattice polytope. $P$ is of degree one if and only if it can be obtained as an iterated pyramid over either a Lawrence prism $L(a_0,\dots, a_{n-1})$ or over an exceptional simplex $\Delta_2:={\rm Conv}\left((0,0),(2,0),(0,2)\right)$.
\end{theorem}

In this section we use the correspondence between polytopes of degree one and varieties of minimal degree to prove three Theorems about toric codes from polytopes of degree one.
First, we focus on dual codes $C_P^*$. We prove that, although the minimum distance of $C_P^*$ is often very small, the mode of size $c+1$ of $C_P^*$ is as large as possible. Moreover, we prove that polytopes of degree one are in fact completely characterized by that property. We then focus on primal codes $C_P$ and determine their parameters for all polytopes of degree one, extending work of Little and Schenck~\cite[Section 3]{LS}. In Table~\ref{T1} we show that for some values of the parameters these toric codes could be of interest in applications.

\begin{theorem} Let $P$ be a polytope of degree one and dimension at least two. For all but finitely many finite fields $\FF_q$ we have that either $\dmin(C_P^*(\FF_q))=3$ or $P=\Delta_2$ and $\dmin(C_P^*(\FF_q))=4$. 
\end{theorem}
\begin{proof}  Let $P$ be a polytope of degree one and dimension at least two. Then either,
\begin{enumerate}
\item{ $X_P$ is a cone or a rational normal scroll or}
\item{ $P=\Delta_2$ and $X_P$ is the Veronese surface in $\PP^5$ (i.e. the image of $\PP^2$ under the morphism $\nu_2$ given by the complete linear system of $\Osh_{\PP^2}(2)$).}
\end{enumerate}
In the first case $X_P$ contains lines and the result follows from Lemma~\ref{THM: Dual} for any sufficiently large field. In the second case note that the set of conics vanishing at four collinear points in $\PP^2$ is $3$-dimensional so their images under $\nu_2$ do not impose independent conditions. This is because, by Bezout«s Theorem, every quadric vanishing at three collinear points must be reducible and have the line spanned by these points as a component showing that the vector space of quadrics through the points is $3$-dimensional. As a result $\dmin(C_P^*(\FF_q))\leq 4$. On the other hand any set $S$ consisting of three distinct points of $\PP^2$ imposes independent conditions on quadrics. This is because either the three points are collinear and we argue as in the previous paragraph or the three points are not collinear and the surface $Y:={\rm Bl}_{S}(\PP^2)\rightarrow \PP^2$ is a Del Pezzo surface where it is immediate to check that $h^0(2H-E_1-E_2-E_3)=4$ (where $H$ is the pullback of a general line under $\pi$ and the $E_i$ are the exceptional divisors). 
\end{proof}

Thus, the codes $C_P^*$ contain very short words. However, the behavior of {\it typical} short words is radically different as shown by the following Theorem,

\begin{theorem}\label{typical} Let $P\subseteq \RR^m$ be any full-dimensional lattice polytope with at least $m+2$ lattice points. For all but finitely many finite fields $\FF_q$ there is a cofinal subset $\mathcal{C}$ of the poset of finite fields containing $\FF_q$ such that for every $\FF_r\in \mathcal{C}$ we have 
\[M^{X_P(\FF_q)}(\codim(X_P)+1,C_P^*(\FF_r))\in\{c+2,c+3\}.\]
Moreover $M^{X_P(\FF_q)}(\codim(X_P)+1,C_P^*(\FF_r))=c+3$ if and only if $P$ is a polytope of degree one.
\end{theorem}
\begin{proof} Follows from Theorem~\ref{Extremality} together with the fact that $X_P$ is a variety of minimal degree if and only if $P$ is a polytope of degree one.
\end{proof}
Next we determine the values of the parameters of the primal codes $C_P$ for polytopes of degree one,
\begin{theorem}\label{Thm: dMinimalDegree} Let $P\subseteq [0,q-1]^m\subseteq \RR^m$ be a lattice polytope of degree one and let $(n,k,\dmin)$ be the length, rank and minimum distance of the code $C_P(\FF_q)$. Then $n=(q-1)^m$ and:
\begin{enumerate} 
\item{ If $P$ is obtained from $\Delta_2\subseteq \RR^2$ by iterated pyramids then 
\[\begin{array}{lll}
k=m+4 & \text{and }& \dmin=(q-1)^m-2(q-1)^{m-1}.\\
\end{array}
\]
}
\item{If for some integers $a_i$, $0< a_0\leq a_1\leq\dots \leq a_{s-1}$, $P$ is obtained from $L(a_0,\dots, a_{n-1})$  by $m-s$ iterated pyramids then
$k=m+\sum_{i=0}^{s-1}a_i$ and 
\[
\dmin=\begin{cases}
(q-1)^m-a_{s-1}(q-1)^{m-1}\text{ , if $a_{s-2}<a_{s-1}$}\\
(q-1)^m- (q-1)^{m-2}\Big((q-1) + a_{s-1} (q-2)\Big)\text{ , if $a_{s-2}=a_{s-1}$.}\\
\end{cases}
\]
}
\end{enumerate}
\end{theorem}
\begin{proof} Since $P$ is an $m$-dimensional polytope, the length of its corresponding toric code is the cardinality of $T(\FF_q)$ so $n=(q-1)^{m}$. The pyramid operation increases the dimension and the number of lattice points by one and by~\cite[Theorem 2.4]{SS} multiplies the minimum distance by $(q-1)$. It is thus sufficient to prove our results for Lawrence prisms and the simplex $\Delta_2$. The claims on the rank are immediate from the fact that $\Delta_2$ has $6$ lattice points and that $L(a_0,\dots, a_{s-1})$ has $s+\sum_{i=0}^{n-1}a_j$ lattice points. Next we claim that if $Q\subseteq [0,q-1]^m$ is any full-dimensional polytope containing a segment $[0,f]e_m$ and contained in the simplex $\Delta_f$ then the minimum distance of $C_Q$ equals $(q-1)^m-f(q-1)^{m-1}$. This occurs because $Q\subseteq \Delta_f$ implies $\dmin(C_Q)\geq \dmin(C_{\Delta_f})$ and because the reducible section $(x-\alpha_1)\dots (x-\alpha_f)\in H^0(Q)$ for distinct $\alpha_{i}\in \FF_q^*$ has exactly $f(q-1)^{m-1}$ zeroes in $T(\FF_q)$ showing that $\dmin(C_Q)\leq (q-1)^m-f(q-1)^{m-1}$. By~\cite[Corollary 3.2]{SS} this last quantity equals $\dmin(C_Q)$ proving the claim. In particular, if $a_{s-2}<a_{s-1}$ then the lawrence prism $L(a_0,\dots, a_{s-1})$ satisfies $[0,a_{s-1}]e_s\subseteq L(a_0,\dots, a_{s-1})\subseteq \Delta_{a_{s-1}}$ proving the first formula in $(2)$ above for the minimum distance. Finally, if $a_{s-2}=a_{s-1}$ then $L:=L(a_0,\dots, a_{s-1})\subseteq \Delta_1\times [0,a_{s-1}]$ where $\Delta_1:={\rm Conv}(e_0,\dots, e_{s-2})\subseteq \RR^{s-1}$. Now by~\cite[Theorem 2.1]{SS} $\dmin( C_{\Delta_1\times [0,a_{s-1}]})=\dmin(C_{\Delta_1})\dmin(C_{[0,a_{s-1}]})$. Both terms are the minimum distances of codes coming from simplices and thus we can explicitly compute them, obtaining
\[
\dmin(C_{L(a_0,\dots, a_{s-1})})\geq ((q-1)^{s-1}-(q-1)^{s-2})(q-1-a_{s-1})
\] 
Now, the section $(x_{s-1}-\alpha_1)\dots (x_{s-1}-\alpha_{a_{s-1}})(x_{s-2}-\beta)\in H^0(L)$ for distinct $\alpha_i,\beta\in \FF_q^{*}$ has precisely 
$(q-1)^{s-1}+(q-1)^{s-2}a_{s-1}(q-2)$ zeroes in $T(\FF_q)$, counted by splitting them according to whether or not $x_{s-2}$ equals $\beta$. Thus the lower bound on $\dmin(C_{L(a_0,\dots, a_{s-1})})$ is achieved, proving the remaining claim.
\end{proof}
\begin{remark} As observed by Little and Schenck, the sections with the highest number of zeroes in $T(\FF_q)$ (i.e. those corresponding to the words which determine the minimum distance) are very often highly reducible. The previous proof shows that this is indeed the case for all degree one polytopes.
\end{remark}

The following table illustrates some parameter values for primal codes from polytopes of degree one.
\begin{table}[htbp]\label{T1}
\caption{Parameter values for primal toric codes from some three-dimensional Lawrence prisms.}	\label{T2}
	\centering
   \begin {tabular}{|c||c|c|c||c|c|c|}	\hline
$q$		&   $a_0$	 &   $a_1$	   &    $a_2$	 &  $n$   &   $k$   & $\dmin$  \\ \hline
7 		&    1      &   2    	   &    3   	 &  216   &    9    & 108  \\ \hline
11 		&    1      &  2    	   &    3   	 &  1000   &    9    & 700  \\ 
11 		&    2      &  4    	   &     7   	 &  10648   &    16    & 300  \\  \hline
23 		&    1      &  6    	   &    9   	 &  10648    &    19    & 6292  \\ 
23 		&    5      &  7    	   &     14   	 &  1000   &    29    & 3872  \\  \hline
47 		&    2      &  4    	   &    7   	 &  97336   &    16    & 82524  \\ 
47 		&    15      &  25    	   &     14   	 &  97336   &    57    & 67712  \\  \hline
\end{tabular}
\end{table}
\begin{remark} The results of this article can be extended from toric codes to the setting of general algebraic geometry codes. We chose the former to make the presentation simpler.
\end{remark}

\end{document}